\documentclass[12pt]{amsart}
\usepackage[margin=1in]{geometry}
\usepackage{amsmath,amsfonts,amsthm,amssymb,bbm}
\usepackage{graphicx,color,dsfont}
\usepackage{enumitem}

\newtheorem{theorem}{Theorem}
\newtheorem{lemma}{Lemma}
\newtheorem{proposition}{Proposition}

\newtheorem{corollary}{Corollary}

\newtheorem{claim}{Claim}

 \theoremstyle{definition}
 
 \theoremstyle{remark}

 \numberwithin{equation}{section}

\newcommand{\vertiii}[1]{{\left\vert\kern-0.25ex\left\vert\kern-0.25ex\left\vert #1
    \right\vert\kern-0.25ex\right\vert\kern-0.25ex\right\vert}}

\newcommand{\f}[2]{\frac{#1}{#2}}

\newcommand{\cl}{{\mathcal L}}

\begin{document}

\newcommand{\abs}[1]{\lvert#1\rvert}

\newcommand{\al}{\alpha}
\newcommand{\be}{\beta}
\newcommand{\wh}[1]{\widehat{#1}}
\newcommand{\ga}{\gamma}
\newcommand{\Ga}{\Gamma}
\newcommand{\de}{\delta}
\newcommand{\ben}{\beta_n}
\newcommand{\De}{\Delta}
\newcommand{\ve}{\varepsilon}
\newcommand{\ze}{\zeta}
\newcommand{\Th}{\Theta}
\newcommand{\ka}{\kappa}
\newcommand{\la}{\lambda}
\newcommand{\laj}{\lambda_j}
\newcommand{\lak}{\lambda_k}
\newcommand{\La}{\Lambda}
\newcommand{\si}{\sigma}
\newcommand{\Si}{\Sigma}
\newcommand{\vp}{\varphi}
\newcommand{\om}{\omega}
\newcommand{\Om}{\Omega}

\newcommand{\ro}{{\mathbf R}}
\newcommand{\rn}{{\mathbf R}^n}
\newcommand{\rd}{{\mathbf R}^d}
\newcommand{\rmm}{{\mathbf R}^m}
\newcommand{\rone}{\mathbf R}
\newcommand{\rtwo}{\mathbf R^2}
\newcommand{\rthree}{\mathbf R^3}
\newcommand{\rfour}{\mathbf R^4}
\newcommand{\ronen}{{\mathbf R}^{n+1}}
\newcommand{\ku}{\mathbf u}
\newcommand{\kw}{\mathbf w}
\newcommand{\kf}{\mathbf f}
\newcommand{\kz}{\mathbf z}

\newcommand{\N}{\mathbf N}

\newcommand{\tn}{\mathbf T^n}
\newcommand{\tone}{\mathbf T^1}
\newcommand{\ttwo}{\mathbf T^2}
\newcommand{\tthree}{\mathbf T^3}
\newcommand{\tfour}{\mathbf T^4}

\newcommand{\zn}{\mathbf Z^n}
\newcommand{\zp}{\mathbf Z^+}
\newcommand{\zone}{\mathbf Z^1}
\newcommand{\zz}{\mathbf Z}
\newcommand{\ztwo}{\mathbf Z^2}
\newcommand{\zthree}{\mathbf Z^3}
\newcommand{\zfour}{\mathbf Z^4}

\newcommand{\hn}{\mathbf H^n}
\newcommand{\hone}{\mathbf H^1}
\newcommand{\htwo}{\mathbf H^2}
\newcommand{\hthree}{\mathbf H^3}
\newcommand{\hfour}{\mathbf H^4}

\newcommand{\cone}{\mathbf C^1}
\newcommand{\ctwo}{\mathbf C^2}
\newcommand{\cthree}{\mathbf C^3}
\newcommand{\cfour}{\mathbf C^4}
\newcommand{\dpr}[2]{\langle #1,#2 \rangle}

\newcommand{\sn}{\mathbf S^{n-1}}
\newcommand{\sone}{\mathbf S^1}
\newcommand{\stwo}{\mathbf S^2}
\newcommand{\sthree}{\mathbf S^3}
\newcommand{\sfour}{\mathbf S^4}

\newcommand{\lp}{L^{p}}
\newcommand{\lppr}{L^{p'}}
\newcommand{\lqq}{L^{q}}
\newcommand{\lr}{L^{r}}
\newcommand{\echi}{(1-\chi(x/M))}
\newcommand{\chip}{\chi'(x/M)}

\newcommand{\wlp}{L^{p,\infty}}
\newcommand{\wlq}{L^{q,\infty}}
\newcommand{\wlr}{L^{r,\infty}}
\newcommand{\wlo}{L^{1,\infty}}

\newcommand{\lprn}{L^{p}(\rn)}
\newcommand{\lptn}{L^{p}(\tn)}
\newcommand{\lpzn}{L^{p}(\zn)}
\newcommand{\lpcn}{L^{p}(\cn)}
\newcommand{\lphn}{L^{p}(\cn)}

\newcommand{\lprone}{L^{p}(\rone)}
\newcommand{\lptone}{L^{p}(\tone)}
\newcommand{\lpzone}{L^{p}(\zone)}
\newcommand{\lpcone}{L^{p}(\cone)}
\newcommand{\lphone}{L^{p}(\hone)}

\newcommand{\lqrn}{L^{q}(\rn)}
\newcommand{\lqtn}{L^{q}(\tn)}
\newcommand{\lqzn}{L^{q}(\zn)}
\newcommand{\lqcn}{L^{q}(\cn)}
\newcommand{\lqhn}{L^{q}(\hn)}

\newcommand{\lo}{L^{1}}
\newcommand{\lt}{L^{2}}
\newcommand{\li}{L^{\infty}}

\newcommand{\co}{C^{1}}
\newcommand{\ci}{C^{\infty}}
\newcommand{\coi}{C_0^{\infty}}

\newcommand{\ca}{\mathcal A}
\newcommand{\cs}{\mathcal S}
\newcommand{\cm}{\mathcal M}
\newcommand{\cf}{\mathcal F}
\newcommand{\cb}{\mathcal B}
\newcommand{\ce}{\mathcal E}
\newcommand{\cd}{\mathcal D}
\newcommand{\cn}{\mathcal N}
\newcommand{\cz}{\mathcal Z}
\newcommand{\crr}{\mathbf R}
\newcommand{\cc}{\mathcal C}
\newcommand{\ch}{\mathcal H}
\newcommand{\cq}{\mathcal Q}
\newcommand{\cp}{\mathcal P}
\newcommand{\cx}{\mathcal X}

\newcommand{\pv}{\textup{p.v.}\,}
\newcommand{\loc}{\textup{loc}}
\newcommand{\intl}{\int\limits}
\newcommand{\iintl}{\iint\limits}
\newcommand{\dint}{\displaystyle\int}
\newcommand{\diint}{\displaystyle\iint}
\newcommand{\dintl}{\displaystyle\intl}
\newcommand{\diintl}{\displaystyle\iintl}
\newcommand{\liml}{\lim\limits}
\newcommand{\suml}{\sum\limits}
\newcommand{\ltwo}{L^{2}}
\newcommand{\supl}{\sup\limits}
\newcommand{\df}{\displaystyle\frac}
\newcommand{\p}{\partial}
\newcommand{\Ar}{\textup{Arg}}
\newcommand{\abssigk}{\widehat{|\si_k|}}
\newcommand{\ed}{(1-\p_x^2)^{-1}}
\newcommand{\tT}{\tilde{T}}
\newcommand{\tV}{\tilde{V}}
\newcommand{\wt}{\widetilde}
\newcommand{\Qvi}{Q_{\nu,i}}
\newcommand{\sjv}{a_{j,\nu}}
\newcommand{\sj}{a_j}
\newcommand{\pvs}{P_\nu^s}
\newcommand{\pva}{P_1^s}
\newcommand{\cjk}{c_{j,k}^{m,s}}
\newcommand{\Bjsnu}{B_{j-s,\nu}}
\newcommand{\Bjs}{B_{j-s}}
\newcommand{\Ly}{L_i^y}
\newcommand{\dd}[1]{\f{\partial}{\partial #1}}
\newcommand{\czz}{Calder\'on-Zygmund}
\newcommand{\chh}{\mathcal H}

\newcommand{\lbl}{\label}
\newcommand{\beq}{\begin{equation}}
\newcommand{\eeq}{\end{equation}}
\newcommand{\beqna}{\begin{eqnarray*}}
\newcommand{\eeqna}{\end{eqnarray*}}
\newcommand{\beqn}{\begin{equation*}}
\newcommand{\eeqn}{\end{equation*}}
\newcommand{\bp}{\begin{proof}}
\newcommand{\ep}{\end{proof}}
\newcommand{\bprop}{\begin{proposition}}
\newcommand{\eprop}{\end{proposition}}
\newcommand{\bt}{\begin{theorem}}
\newcommand{\et}{\end{theorem}}
\newcommand{\bex}{\begin{Example}}
\newcommand{\eex}{\end{Example}}
\newcommand{\bc}{\begin{corollary}}
\newcommand{\ec}{\end{corollary}}
\newcommand{\bcl}{\begin{claim}}
\newcommand{\ecl}{\end{claim}}
\newcommand{\bl}{\begin{lemma}}
\newcommand{\el}{\end{lemma}}
\newcommand{\dea}{(-\De)^\be}
\newcommand{\naa}{|\nabla|^\be}
\newcommand{\cj}{{\mathcal J}}

\title[ Ground states for the Hartree problem ]
 {On the classification of the spectrally stable  standing waves of  the Hartree problem}
\author[V. Georgiev]{Vladimir Georgiev}

\address{%
Department of Mathematics \\
University of Pisa \\
 Largo Bruno Pontecorvo 5
 I - 56127 Pisa
 \\ Italy \\ and \\
 Faculty of Science and Engineering \\ Waseda University \\
 3-4-1, Ohkubo, Shinjuku-ku, Tokyo 169-8555 \\
Japan}

\email{georgiev@dm.unipi.it}

\thanks{ Georgiev is supported in part by  INDAM, GNAMPA - Gruppo Nazionale per l'Analisi Matematica, la Probabilita e le loro Applicazioni, by Institute of Mathematics and Informatics, Bulgarian Academy of Sciences and Top Global University Project, Waseda University.  Stefanov    is partially  supported by  NSF-DMS, Applied Mathematics program under grants \# 1313107 and \# 1614734.}

\author[A. Stefanov]{\sc Atanas Stefanov}
\address{ Department of Mathematics,
University of Kansas,
1460 Jayhawk Boulevard,  Lawrence KS 66045--7523, USA}
\email{stefanov@ku.edu}



\date{}

\begin{abstract}

We consider the fractional Hartree model, with general power non-linearity and space dimension. We construct variationally the ``normalized''
solutions for the corresponding Choquard-Pekar model - in particular a number of key properties, like smoothness and bell-shapedness are established.
As a consequence of the  construction, we show that these solitons are spectrally stable as solutions to the time-dependent Hartree model.

In addition, we  analyze the spectral stability of the Moroz-Van Schaftingen solitons
of the classical Hartree problem, in any dimensions and power non-linearity.
A full classification is obtained, the main conclusion of which is   that only and exactly the ``normalized'' solutions
(which exist only in a portion of the range) are spectrally stable.

\end{abstract}

\maketitle
\section{ Introduction}
We consider the Cauchy problem for the Hartree equation
\begin{equation}
 \label{1010}
  \left\{
  \begin{array}{l}
  i u_t  +(-\De)^\be u - v |u|^{p-2} u=0, \ \ (t,x)\in \rone_+\times \rd \\
  (-\De)^{\al/2} v = |u|^p, \\
  u(0,x)=u_0(x).
  \end{array}
  \right.
\end{equation}
Here, the operator $(-\De)^\be$ is defined via Fourier multiplier with $|2\pi \xi|^{2\be}$, see the relevant definitions\footnote{Sometimes we shall use the notation
$|\nabla|=\sqrt{-\De}$. }  below in Section \ref{sec:prelim}. Unless otherwise indicated in a particular place, the values of the parameters will be henceforth  as follows: $ \be \in (0,1], d\geq 1, p>1, \al \in (0,d)$.
 Resolve the elliptic equation
$$
v=(-\De)^{-\al/2}[|u|^p]=I_\al[|u|^p]=I_{\al}(\cdot)*|u|^p, I_\al(x)= \f{\Ga(\f{d-\al}{2})}{\Ga(\f{\al}{2})\pi^{d/2} 2^{\al}|x|^{d-\al}}.
$$
We obtain the system
 \begin{equation}
 \label{10}
  \left\{
  \begin{array}{l}
  i u_t  +(-\De)^\be u -  c_{d, \ga} [|\cdot|^{-\ga}*|u|^p] |u|^{p-2} u=0, \ \ (t,x)\in \rone_+\times \rd \\
  u(0,x)=u_0(x),
  \end{array}
  \right.
\end{equation}
where we have introduced the parameter $\ga:=d-\al\in (0,d)$.

 We will be interested in the properties of standing wave solutions $u(t,x)=e^{i \om t} \phi(x)$,  with $\phi>0$. Clearly,  $\phi=\phi_{p,\om}$   will then  satisfy  the profile equation
\begin{equation}
\label{20}
(-\De)^\be \phi - c_{d,\ga}(|\cdot|^{-\ga}*|\phi|^{p}) |\phi|^{p-2}\phi = \om \phi, \ \ x\in \rd.
\end{equation}
The equation \eqref{20} is (a fractional) version of the well-known Choquard equation. This is  a good point for us to review some of the developments in the classical theory for this model.
\subsection{The classical Hartree-Choquard-Pekar model}
 As one expects, most of the work was done in the classical context, $\be=1$, for the Hartree-Choquard-Pekar system (for $\al\in (0,d)$)
\begin{equation}
\label{t:10}
\begin{array}{l}
i u_t -\De u -  I_\al[|u|^p] |u|^{p-2} u=0, (t,x)\in \rone\times \rd.
\end{array}
\end{equation}
The standing wave solutions of the form $e^{-i t} \vp$ satisfy
\begin{equation}
\label{t:20}
-\De \vp+\vp - I_\al[|\vp|^p] |\vp|^{p-2} \vp=0.
\end{equation}
The question for existence  of localized solutions for \eqref{t:20} has been well-studied in the last thirty years or so, mostly for special values of the parameters. For example, the case of $p=2$ has been studied in \cite{Lieb, PLLions, menzala} by the variational approach, and in \cite{TM} by ODE techniques.
 The case $\gamma=d-2 \geq 1$ and $2 \leq p < (2d-\gamma)/(d-2), $ was previously considered in \cite{GV} by introducing the constraint minimizer similar  to the one of Section 3 below.

Quite recently, a general classification result for such solutions was put forward in \cite{MZ} and a complete proof was presented in \cite{MVS1}.  The following theorem is a summary of the results presented in Theorems 1, 2, 3 in \cite{MVS1}.
\begin{theorem}
\label{vs:10}
Let $d\geq 1, \al \in (0,d)$ and $p\in (1, \infty)$.

Assuming $\f{d-2}{2d-\ga}<\f{1}{p}<\f{d}{2d-\ga}$, there is a solution $\vp\in H^2(\rd)\cap L^\infty(\rd)$ of \eqref{t:20}. Moreover, these  solutions are found in the form $\vp=t_0 \Phi$, where
$\Phi$ is a minimizer of the following optimization problem
\begin{equation}
\label{200}
\inf\limits_{u\neq 0}\f{\int_{\rd} [|\nabla u(x)|^2+|u(x)|^2] dx }{\dpr{I_\al[|u|^p]}{|u|^p}^{1/p}}
\end{equation}
and the scalar $t_0$ is selected so that $\int_{\rd} [|\nabla \vp(x)|^2+|\vp(x)|^2] dx=\dpr{I_\al[\vp^p]}{\vp^p}$.

In addition, there exists $x_0\in \rd$ and a decreasing function $\rho:\rone_+\to \rone_+$,  so that $\vp(x)=\pm \rho(|x-x_0|)$. Finally, $\vp$ satisfies the Pohozaev's identity
\begin{equation}
\label{pohoz}
(d-2) \int_{\rd} |\nabla \vp|^2+d \int_{\rd} |\vp|^2 =\f{2d-\ga}{p} \dpr{I_\al[\vp^p]}{\vp^p}.
\end{equation}

In  the complementary range: $\f{1}{p}\leq \f{d-2}{2d-\ga}$ or $\f{1}{p}\geq \f{d}{2d-\ga}$, the only regular and localized solution of \eqref{t:20} is $u=0$.
\end{theorem}
We note that functions in the form $u(x)=\rho(|x|)$, where $\rho:\rone_+\to \rone_+$ is a decreasing and vanishing at infinity,   are called {\it bell-shaped}.  An equivalent way of expressing the same  property is the equality $u=u^*$, where $u^*$ is the decreasing rearrangement in the sense of Riesz. In any case,  the statement of Theorem \ref{vs:10} implies that the only solutions of \eqref{t:20} are translates of bell-shaped functions.

\subsection{The Klein-Gordon-Hartree model}
We also consider the related Klein-Gordon-Hartree model\footnote{Here, we could have considered the more general fractional version of the model, similar to \eqref{1010}.
Since we can only present a complete stability classification only in the case $\be=1$, we prefer to state only the classical form, \eqref{410}.}
but we have preferred to just consider the classical
\begin{equation}
\label{410}
u_{tt}- \De u+u -  I_\al[| u |^p] | u |^{p-2} u =0.
\end{equation}
Standing wave solutions in the form $u(t,x) = e^{i \om t} \Psi$ must of course satisfy the elliptic PDE
\begin{equation}
\label{420}
-\De \Psi +(1-\om^2) \Psi - I_\al[| \Psi |^p] | \Psi |^{p-2} \Psi =0,
\end{equation}
which is of course closely related to \eqref{t:20}, provided $|\om|<1$, which we assume henceforth. A simple rescaling argument, together with Theorem \ref{vs:10}, allows us to conclude that there are bell-shaped
solutions of \eqref{420} in the form
\begin{equation}
\label{eq:psi}
\Psi(x)=(1-\om^2)^{\f{2+\al}{4(p-1)}}\vp(x\sqrt{1-\om^2}),
\end{equation}
where $\vp$ is the set of solutions described in  Theorem \ref{vs:10}.

\subsection{Main results}
Our results concern both the fractional model \eqref{10} and the more classical version \eqref{t:10}.  More precisely, we are interested in the existence properties of solitary waves for \eqref{10}, that is whether and under what conditions, one obtains nice ground state solutions of \eqref{20}.
\subsubsection{The fractional Choquard equation - existence and stability}
This calls for a  generalization of Theorem \ref{vs:10} above, at least in the existence part of it.
We have the following existence result.
  \begin{theorem}
  \label{theo:10}
Let  $\be\in (0,1], \ga\in (0,d)$ and $p>1$. Assume in addition the relationship
  \begin{equation}
  \label{50}
  0<(p-2) d+\ga<2\be.
  \end{equation}
  Then,  there exists a solution of \eqref{20}, $\phi$,  namely a solution of a constrained minimization problem  \eqref{30} below.  Moreover, $\phi$
  is bell-shaped.
  \end{theorem}
Note that the inequality $0 < (p-2) d+\ga$ is exactly equivalent to the requirement $\f{1}{p}<\f{d}{2d-\ga}$ from Theorem \ref{vs:10}. The other inequality however, say for the classical case $\be=1$, is $p<2+\f{2-\ga}{d}$, which is a strict subset of the requirement $\f{d-2}{2d-\ga}<\f{1}{p}$ imposed in Theorem \ref{vs:10}.  So, we do not seem to get all the solitary waves in this way, more on this point below.
In fact, this brings us to our second object of interest, namely the stability of the waves constructed in Theorem \ref{theo:10}. It turns out that the waves constructed in Theorem \ref{theo:10} are spectrally stable as solutions of \eqref{10}. More precisely, we have the following result.
\begin{theorem}
   \label{a:110}
   Let $p>2$. Then, the ground states $\phi$ constructed in Theorem \ref{theo:10}  are spectrally stable as solutions of \eqref{10}.
   \end{theorem}
   {\bf Remark:} The condition $p>2$ appears to be of a technical nature and it is likely removable, if one knows extra information about the waves constructed in Theorem \ref{theo:10} - similar to Lemma \ref{le:2} and \ref{le:vlado} below.

 The waves constructed in Theorem   \ref{theo:10} are constructed as  the minimizers of the problem   $\inf\limits_{\|u\|_{L^2}=\la} E(u)$ (dubbed ``normalized solutions'' in \cite{MVS1}), where the energy functional is given by
 $$
E(u):=\f{1}{2} \int_{\rd} [|\nabla u(x)|^2+|u(x)|^2] dx  - \f{c_{d,\ga}}{2p}  \int_{\rd\times \rd }
  \f{|u(x)|^{p} |u(y)|^{p}}{|x-y|^\ga} dx dy.
$$
They  turn out to be spectrally stable, per the claim of Theorem \ref{a:110}.  It so happens {\it these are all the stable solitary waves there are}, at least in the classical case $\be=1$, as we discuss now.

\subsubsection{The classical Choquard equation - classification of the stable ground states}

In the classical Choquard case, \eqref{t:20}, we provide  a full description of the localized and regular solutions, given in Theorem \ref{vs:10}. A natural question is then: which of these waves are spectrally stable as solutions of \eqref{t:10}? The full classification is provided in the following theorem.
\begin{theorem}
\label{theo:60}
Let $d\geq 1, \al \in (0,d)$ and $p\in (1, \infty)$, so that $\f{d-2}{d+\al}<\f{1}{p}<\f{d}{d+\al}$.

Consider any solution $\vp$ of \eqref{t:20} guaranteed by Theorem \ref{vs:10}. Then, the solution $e^{- i t} \vp$ of the time dependent Hartree problem  \eqref{t:10} is spectrally stable \underline{if and only if}
$$
\Ga:=2-\ga-(p-2) d=2+\al-(p-1) d \geq 0.
$$
More specifically,
\begin{itemize}
\item If $d=1,2$, the MVS waves are stable if and only if
$
1+\f{\al}{d}<p\leq 1+\f{2+\al}{d}
$
and unstable in the complementary range $1+\f{2+\al}{d}<p<\infty$. The instability presents itself as a   simple  growing mode.
\item If $d\geq 3$, the MVS waves are stable if
$
1+\f{\al}{d}<p\leq 1+\f{2+\al}{d}
$
and unstable in the complementary range
$1+\f{2+\al}{d}<p<1+\f{2+\al}{d-2}$. The instability presents itself as a   simple  growing mode.
\end{itemize}
\end{theorem}
A few remarks are in order.
\begin{enumerate}
\item Note that the statement in Theorem \ref{theo:60} agrees well with Theorem \ref{theo:10}. In particular, we find in Theorem \ref{theo:60} that the only stable solitons for the Hartree model are the normalized solutions - that is, those obtained in the range $\Ga=2+\al-(p-1) d\geq 0$.
\item Some of the instability results have been previously established by other methods. In particular, in the case $d=3$ and in the optimal range
$\f{5+\al}{3}<p<3+\al$, strong instability was established in \cite{CG}, see also a very recent extension of these results to Hartree models with potentials in \cite{CSM}. In these works, the authors employ  a virial identity type arguments, which show that there exists
data arbitrarily close to the soliton, for which the solution blows up in finite time.

\item
Note that in the limit $\al\to 0+$, we  recover the stability results for the NLS model (with power non-linearity of order $q=2p-1$). This is indeed the case, since (formally as $\al\to 0+$) one obtains stability for $1<p<1+\f{2}{d}$, which is equivalent to $q=2p-1\in (1,1+\f{4}{d})$, the well-known Schr\"odinger result.
\item In the case $\Ga=0$ or equivalently $p=1+\f{2+\al}{d}$, we discover that there is an extra pair of  elements in the generalized kernel of
$L_+$, $gKer[L_+]$.  Indeed, using $p$ as a bifurcation parameter, one sees that starting from $p> 1+\f{2+\al}{d}$, there is a pair of stable/unstable eigenvalues which approaches the origin and it turns into a  pair of purely imaginary eigenvalues for $p<1+\f{2+\al}{d}$. At $p=1+\f{2+\al}{d}$, this pair introduces extra two dimensions  in   $gKer[L_+]$.   This is very similar to the pseudo-conformal symmetry for the standard Schr\"odinger equation, which arises only for $p=1+\f{4}{d}$.   We thus conjecture that there is an extra symmetry, for this particular case, which generates this extra algebraic multiplicity of the zero eigenvalue.
 \end{enumerate}
 Our next result is a complete characterization of the  spectral stability for the waves $e^{i \om t} \Psi$ of \eqref{420}  in the Klein-Gordon-Hartree context.
 \begin{theorem}
 \label{theo:50}
 Let $d\geq 1, \al \in (0,d), \om \in (-1,1)$, $p\in (1, \infty)$ and $\f{d-2}{d+\al}<\f{1}{p}<\f{d}{d+\al}$.

Let $\vp$ is a MVS solution of  \eqref{t:20}, which  exists in the specified range of $p$ according to Theorem \ref{vs:10}.   Then, the solution $e^{i \om t} \Psi(x)$ described in \eqref{eq:psi}  of the time dependent Klein-Gordon-Hartree problem  \eqref{410} is spectrally stable \underline{if and only if}
$$
\Ga>0,  \sqrt{\f{p-1}{p-1+\Ga}}=\sqrt{\f{p-1}{2+\al-(p-1)(d-1)}}<|\om|<1.
$$
More precisely,
\begin{itemize}
\item If $d=1,2$, the MVS waves are unstable, if $1+\f{2+\al}{d}<p<\infty$ or $p\in (1+\f{\al}{d}, 1+\f{2+\al}{d}]$ and $0\leq |\om|<\sqrt{\f{p-1}{2+\al-(p-1)(d-1)}}$. Equivalently, the waves are stable, exactly when
$$
p\in (1+\f{\al}{d}, 1+\f{2+\al}{d}), \sqrt{\f{p-1}{2+\al-(p-1)(d-1)}}<|\om|<1.
$$
\item If $d\geq 3$, the waves are unstable for $p\in (1+\f{2+\al}{d}, 1+\f{2+\al}{d-2})$ or $p\in (1+\f{\al}{d}, 1+\f{2+\al}{d}]$ and $0\leq |\om|<\sqrt{\f{p-1}{2+\al-(p-1)(d-1)}}$. Equivalently, stability occurs exactly for
$$
p\in (1+\f{\al}{d}, 1+\f{2+\al}{d}),\ \  \sqrt{\f{p-1}{2+\al-(p-1)(d-1)}}<|\om|<1.
$$
\end{itemize}
 \end{theorem}

\section{Preliminaries}
\label{sec:prelim}
The Fourier transform and its inverses are taken to be in the form
$$
\hat{f}(\xi)=\int_{\rd} f(x) e^{-2\pi i x\cdot \xi} dx, f(x) =\int_{\rd} \hat{f}(\xi) e^{2\pi i x\cdot \xi} d\xi.
$$
The operators $(-\De)^\be$ are defined through their multipliers (acting on Schwartz functions $f\in \cs$) as follows
$$
\widehat{(-\De)^\be f}(\xi)=|2\pi \xi|^{2 \be} \hat{f}(\xi).
$$
Note that sometimes, we will use instead the Zygmund operator $|\nabla|:=\sqrt{-\De}$.
We   make heavy use of the symmetric decreasing rearrangements of a function $f$, denoted by $f^*$ .
This is a classical object, see for example \cite{LL}, Chapter 3.  In that regard, recall  that $\|f^*\|_{L^p}=\|f\|_{L^p}$, for $1\leq p\leq \infty$ .  In addition, we make use  of the classical inequality
  \begin{equation}
  \label{RR}
   \int_{\rd} f(x) g(x) dx \leq \int_{\rd} f^*(x) g^*(x) dx,
   \end{equation}
  for any non - negative functions $f,g$ decaying sufficiently rapidly at infinity (see Theorem 3.4 in \cite{LL}). If $f$ is a strictly symmetric decreasing function then we have
  equality in \eqref{RR} only if  $g=g^*$ a.e.
 A more sophisticated version of \eqref{RR} is the Riesz's rearrangement inequality (see Theorem 3.7, \cite{LL})
 \begin{equation}
 \label{R}
  \int_{\rd\times \rd} f(x) g(x-y) h(y) dx dy\leq  \int_{\rd\times \rd} f^*(x) g^*(x-y) h^*(y) dx dy.
 \end{equation}
If one of the functions is in fact strictly symmetric decreasing, then equality in \eqref{R} is possible if the other two functions are a fixed
 translate of a symmetric decreasing function.

 The Polya-Szeg\"o inequality states that $\|\nabla  u\|_{L^2(\rd)}\geq \|\nabla u^*\|_{L^2(\rd)}$.
 Here, we  present an extension of this inequality for fractional gradients. This is a relatively recent result. In fact, there is a proof of this fact in \cite{FL}, using completely monotone maps. Here we present a simpler proof, based on the representation of $\naa f$
  in terms of averages of the standard heat kernel operators $e^{t \De}$.
  \begin{proposition}
  \label{prop:10}
  Let $\be \in (0,1]$, $d\geq 1$.  Then, for all functions $u\in \dot{H}^\be$, we have that its decreasing rearrangement $u^*\in \dot{H}^\be$ and moreover
  \begin{equation}
  \label{PS}
  \|\naa u\|_{L^2(\rd)}\geq \|\naa u^*\|_{L^2(\rd)}.
  \end{equation}
  In addition, equality is achieved if and only if there exists $x_0\in \rd$ and a decreasing  function $\rho:\rone_+ \to \rone_+$, so that $u(x)=\rho(|x-x_0|)$.
  \end{proposition}
  {\bf Note:} The classical Polya-Szeg\"o inequality is the particular case $\be=1$.
  \begin{proof}
  Let $\be<1$ and define
  $$
  c_\be:=\int_0^\infty \f{1-e^{-y}}{y^{1+\be}} dy.
  $$
  Setting $y=4\pi^2 |\xi|^2 t$, we have the representation
  $$
  (2\pi |\xi|)^{2\be}=\f{1}{c_\be} \int_0^\infty \f{1-e^{-4\pi^2 t|\xi|^2}}{t^{1+\be}} dt.
  $$
 Equivalently
  $$
  |\nabla|^{2\be}=\f{1}{c_\be} \int_0^\infty \f{1-e^{t\De}}{t^{1+\be}} dt.
  $$
 Since $e^{t\De} f= K_t*f$ and  $K_t(x)=(4\pi t)^{-d/2} e^{-|x|^2/(4t)}$ is strictly symmetric decreasing, we have by \eqref{R} that
 $\dpr{e^{t \De} u}{u}=\dpr{K_t*u}{u}\leq \dpr{K_t*u^*}{u^*}=\dpr{e^{t \De} u^*}{u^*}$ and equality is achieved only if $u(x)=\rho(|x-x_0|)$ for a decreasing function $\rho:\rone_+\to \rone_+$ and $x_0\in\rd$. Thus,
 \begin{eqnarray*}
 \||\nabla|^\be u\|^2 &=& \dpr{|\nabla|^{2\be} u}{u}=\f{1}{c_\be} \int \f{\dpr{u}{u}-\dpr{e^{t \De} u}{u}}{t^{1+\be}} dt\geq \f{1}{c_\be} \int \f{\dpr{u^*}{u^*}-\dpr{e^{t \De} u^*}{u^*}}{t^{1+\be}} dt=\\
 &=& \dpr{|\nabla|^{2\be} u^*}{u^*}=\||\nabla|^\be u^*\|^2.
 \end{eqnarray*}
 Moreover, equality is possible, only if $u(x)=\rho(|x-x_0|)$, as explained above.

  \end{proof}

\section{Existence and  properties of the solutions to the fractional Hartree model}
\label{sec:3}
  For $\la>0$, introduce  the optimization problem
 \begin{equation}
  \label{30}
  \left\{
  \begin{array}{l}
  E(u):=\f{1}{2} \|\naa u\|_{L^2(\rd)}^2 - \f{c_{d,\ga}}{2 p} \int_{\rd\times \rd }
  \f{|u(x)|^{p} |u(y)|^{p}}{|x-y|^\ga} dx dy \to \min \\
  \textup{subject to} \ \ \int_{\rd} |u(x)|^2  dx = \la.
  \end{array}
  \right.
  \end{equation}
At least formally, one can see that the associated Euler-Lagrange equation\footnote{ for potential minimizers $\phi$, whose existence is still to be established} is exactly \eqref{20}. We are now ready to proceed with the proof of  the existence result in Theorem \ref{theo:10}.
\subsection{Existence of solutions for the constrained minimization problem}
More precisely, we have the following.
\begin{proposition}
\label{prop:24}
Let  $\be\in (0,1], \ga\in (0,d)$ and $p>1$ and the relation \eqref{50} holds. Then, the optimization problem \eqref{30} has a bell-shaped solution $\vp$. Moreover, for every solution $u_0$ of \eqref{30}, there exists $x_0\in \rd$, so that $u_0=\pm \rho(|x-x_0|)$, where $\rho:\rone_+\to \rone_+$ is a decreasing and vanishing function. Finally, for every $\la>0$, $E_\la=\inf\limits_{\|u\|^2=\la} E(u)=E(\vp)<0$.
\end{proposition}
  \begin{proof}(Theorem \ref{theo:10})
  First, we show that the constrained minimization problem \eqref{30} is well-posed. That is, the quantity $E(u)$ is bounded from below, when $u$ obeys the constraint $\int_{\rd} |u(x)|^2 dx=\la>0$. To this end, note that we can interpret the potential energy term (or Hartree interaction term) as follows
  $$
  \int_{\rd\times \rd }
  \f{|u(x)|^{p} |u(y)|^{p}}{|x-y|^\ga} dx dy=\dpr{|\cdot|^{-\ga}*|u|^{p}}{|u|^{p}}.
  $$
  Thus, by H\"older's and  the Hardy-Littlewood-Sobolev inequalities, we have
  $$
  \dpr{|\cdot|^{-\ga}*|u|^{p}}{|u|^{p}}\leq \|u\|_{L^{ p r'}}^{2 p}  \||\cdot|^{-\ga}\|_{L^{\f{d}{\ga}, \infty}}=
  C_d \|u\|_{L^{ p r'}}^{2 p},
  $$
  with $r=\f{2d}{\ga}$. Denote $q=p r'$. One can check that $q\geq 2$  is equivalent to the constraint $(p-2)d+\ga\geq 0$, which is one  of the requirements in   \eqref{50}.
  By Sobolev embedding and the Gagliardo-Nirenberg's inequalities, we have
  $$
\|u\|_{L^{pr'}}=  \|u\|_{L^{q}}\leq C_d \|u\|_{\dot{H}^{s}}
  \leq C_d
  \|u\|_{\dot{H}^\be}^{\f{s}{\be}} \|u\|_{L^2}^{1- \f{s}{\be}},
  $$
  where $s=d(\f{1}{2}-\f{1}{q})$, provided $s<\be$ (still to be verified under \eqref{50}).  In  turn, this yields
  \begin{equation}\label{eq.GN1}
     c_{d,\ga} \dpr{|\cdot|^{-\ga}*|u|^{p}}{|u|^{p}}\leq C_{d}  \|\naa  u\|_{L^2}^{\f{2 p s}{\be}} \|  u\|_{L^2}^{2p-\f{2 p s}{\be}},
  \end{equation}
  so we have
  $$
  c_{d,\ga} \dpr{|\cdot|^{-\ga}*|u|^{p}}{|u|^{p}}\leq C_{d, \la}  \|\naa  u\|_{L^2}^{\f{2 p s}{\be}}.
  $$
  Now, the right-hand side of the constraint \eqref{50} ensures exactly that $\f{2 p s}{\be}<2$, so in particular $s<\be$ (since $p >1$), which was required earlier.
  Hence, by Young's inequality
  $$
  E(u)\geq \f{1}{2} \|\naa  u\|^2- C_{d, \la}  \|\naa u\|_{L^2}^{\f{2 p s}{\be}}\geq M_{d,\la},
  $$
  which is the desired control from below of the cost functional $J$.
Introduce
  $$
 E_\la= \inf\limits_{\int u^2(x) dx = \la}  \f{1}{2} \|\naa  u\|_{L^2}^2 - \f{c_{d,\ga}}{2 p} \int_{\rd\times \rd }
  \f{u^{p}(x) u^{p}(y)}{|x-y|^\ga} dx dy,
  $$
  which we know from our previous arguments exists.

  Next, we discuss the existence and the other properties of the constrained minimizers.
  We work with a fixed $\la$, so we omit the superscript in $\phi^\la$.
 Take a minimizing sequence, say $u_n$, and
  $\lim_n E(u_n)= E_\la$.  We
  have by the Polya-Szeg\"o inequality, \eqref{PS}
  $$
  \|\naa u_n \|_{L^2(\rd)}\geq \| \naa u_n^*\|_{L^2(\rd)}.
  $$
  In addition, we have by the Riesz rearrangement inequality
  \begin{eqnarray*}
  \int_{\rd\times \rd} \f{|u_n(x)|^{p} |u_n(y)|^{p}}{|x-y|^\ga} dx dy &= & \dpr{|\cdot|^{-\ga}*u_n^{p}}{u_n^p}\leq \dpr{|\cdot|^{-\ga}*(u^*_n)^p}{(u^*_n)^p}
  = \\
  &= & \int_{\rd\times \rd} \f{(u_n^*)^p(x) (u_n^*)^p(y)}{|x-y|^\ga} dx dy.
  \end{eqnarray*}
  Combining the last two estimates tells us that $E(u_n)\geq E(u_n^*)$, while $\|u_n^*\|_{L^2}^2=\la$.
   Hence, $\lim_n  E(u^*_n)= E_\la$ and $u_n^*$ is uniformly bounded sequence in $H^1(\rd)$.

   Moreover, $u_n^*$ are now bell-shaped functions in the unit sphere of $L^2$, so they have a weakly convergent subsequence (denoted again $u_n^*$), converging weakly in $L^2$ to say $\phi$, a bell-shaped function. By the lower semi-continuity of the norm with respect to weak convergence, $\|\phi\|^2\leq \la$ and also  (note that $\naa u_n^* $ converges weakly to $\naa \phi$)
  \begin{equation}
  \label{vlad:10}
  \liminf \|\naa u_n^*]\|_{L^2}^2\geq \|\naa \phi \|_{L^2}^2.
  \end{equation}
 We also  have that for every $x: |x|>0$,
 $$
 \la=\int_{\rd} |u_n^*(y)|^2 dy\geq \int_{|y|<|x|} |u_n^*(y)|^2 dy\geq c_d |\cdot|^{d} |u_n^*(x)|^2,
 $$
  whence $|u_n^*(x)|\leq C_d |x|^{-d/2}$ for  every $ x \in \rd, x\neq 0$.

 It follows that  $\{u_n^*\}$ is a compact sequence in any $L^q, q>2$ (Relich-Kondrashov's), hence we can assume (after taking subsequences)  $\lim_n \|u_n^*- \phi\|_{L^q}=0$ for any $q>2$.
 As a consequence, we claim that
  \begin{equation}
  \label{vlad:20}
  \lim_n \int_{\rtwo} \f{(u_n^*)^p(x) (u_n^*)^p(y)}{|x-y|^\ga} dx dy = \int_{\rtwo} \f{\phi^p(x) \phi^p(y)}{|x-y|^\ga} dx dy.
  \end{equation}
  Indeed, denoting $K(u,v):= \int_{\rd\times \rd} \f{|u(x)|^p |v(y)|^p}{|x-y|^\ga} dx dy$, we have by triangle inequality and the Hardy-Littlewood-Sobolev inequalities displayed earlier
  \begin{eqnarray*}
  & &   |K(u_n^*,u_n^*)-K(\phi, \phi)|\leq |K(u_n^*,u_n^*)-K(\phi, u_n^*)|+|K(\phi,u_n^*)-K(\phi, \phi)|= \\
  &=& |K(u_n^*-\phi,u_n^*)|+|K(\phi,u_n^*-\phi)|\leq C \|u_n^*-\phi\|_{L^q}(\|u_n^*\|_{L^q}+\|\phi\|_{L^q}).
  \end{eqnarray*}
  for $q=p r'=\f{2d p}{2d-\ga}>2$. Clearly now $\lim_n K(u_n^*,u_n^*)=K(\phi, \phi)$.

  All in all, it follows that $E(\phi)\leq E_\la$. Let us now show that under the constraint \eqref{50}, we have that $E_\la<0$. To that end, take a test function, say $\vp:\int_{\rd} \vp^2(y) dy=\la$, $\ve:\ve<<1$  and set
  $u_\ve=\ve^{d/2} \vp(\ve x)$. Clearly $\|u_\ve\|_{L^2}^2=\la$, so it satisfies the constraint. On the other hand
  \begin{eqnarray*}
  E(u_\ve) &=& \f{1}{2} \|\naa u_\ve\|_{L^2}^2 - \f{c_{d,\ga}}{2 p} \int_{\rd\times \rd}
  \f{|u_\ve(x)|^{p} |u_\ve(y)|^{p}}{|x-y|^\ga} dx dy = \\
  &=& \ve^{2\be} \f{1}{2} \|\nabla \vp \|_{L^2}^2 - \ve^{(p-2) d+\ga} \f{c_{d,\ga}}{2 p}
  \int_{\rd\times \rd}
  \f{|\vp(x)|^{p+1} |\vp(y)|^{p+1}}{|x-y|^\ga} dx dy
  \end{eqnarray*}
  Clearly, since $(p-2) d+\ga<2 $, we have that for small enough $\ve$ the potential term $K(\vp, \vp)$ dominates and hence   $I_\la<0$.

  We are now ready to prove that $\phi$ is a minimizer.  We need to show that $\|\phi\|_{L^2}^2=\la$. Assume that $ \|\phi\|_{L^2}^2<\la$. Then, there is $\mu>1$, so that
  $\|\mu \phi\|_{L^2}^2=\la$. Hence
  \begin{eqnarray*}
  E_\la\leq E(\mu \phi) &=&
  \mu^2\left[\f{1}{2}\|\naa \phi \|^2 - \mu^{2p-2}\f{c_{d,\ga}}{2 p} \int_{\rd\times \rd} \f{\phi^{p}(x) \phi^{p}(y)}{|x-y|^\ga} dx dy\right]\leq \\
  &\leq & \mu^2\left[\f{1}{2}\|\naa \phi  \|^2 - \f{c_{d,\ga}}{2 p} \int_{\rd\times \rd } \f{\phi^{p}(x) \phi^{p}(y)}{|x-y|^\ga} dx dy\right]\leq \mu^2 E_\la,
  \end{eqnarray*}
  a contradiction, since $E_\la<0$. This means $\|\phi\|_{L^2}^2=\la$, but we will show now that  $E(\phi)\leq E_\la$. Indeed, by \eqref{vlad:10} and \eqref{vlad:20},
  $$
  E_\la=\lim_n E(u_n^*)\geq E(\phi).
  $$
  From here, we may conclude that $E(\phi)=E_\la$, otherwise $E(\phi)<E_\la$, a contradiction with the definition of
  $E_\la$.    Thus, $\phi$ is a minimizer.

  Note that in addition, this last equality implies\footnote{otherwise, one gets the impossible inequality $E(\phi)<E_\la$} $  \liminf \|\naa u_n^*]\|_{L^2} = \|\naa \phi \|_{L^2}$, which in addition to the weak  convergence
  $|\nabla|^\beta u_n^*\rightharpoonup |\nabla|^\beta \phi$ allows us to conclude  \\ $\lim_n \||\nabla|^\beta u_n^*  - |\nabla|^\beta \phi\|_{L^2}=0$. So, in the end, it turns out that the minimization sequence converges strongly to the minimizer $\phi$.

  \end{proof}
  Now that we have established the existence of the constrained minimizers, we proceed to our next result which concerns the Euler-Lagrange equation and explicit calculations of  various quantities associated with the energy functional $E(\phi)$.
  \subsection{The Euler-Lagrange equation and scaling relations} For convenience, we introduce the positive parameter
  $$
  \Ga=\Ga_{\ga, \be, d,p}:=2\be -\ga- d (p-2),
  $$
  which appears often in the subsequent formulas.
  \begin{theorem}
  \label{theo:20}
 Under the assumption of Theorem \ref{theo:10} for the parameters, a constrained minimizer $\phi_\la$ as minimizer of \eqref{30}   satisfies the Euler-Lagrange equation \eqref{20}. Moreover, there are the identities
 \begin{eqnarray}
 \label{55}
  \phi_\la(x) &=& \la^{\f{\Ga+(p-1)d}{2\Ga}} \phi(\la^{\f{p-1}{\Ga}} x), \\
 \label{60}
E_\la &=&  \la^{1+\f{2\be (p-1)}{\Ga}} E_1,\\
 \label{62}
J_\la=\|\naa  \phi_\la\|^2 &=& \f{2(\ga+d(p-2))}{\Ga}(-E_1)  \la^{1+\f{2\be (p-1)}{\Ga}}, \\
 \label{64}
 K= c_{d,\ga} \int_{\rd\times \rd} \f{\phi_\la^{p}(x) \phi_\la^{p}(y)}{|x-y|^\ga} dx dy & = & \f{4\be p}{\Ga} (-E_1) \la^{1+\f{2\be (p-1)}{\Ga}},\\
 \label{67}
 \om  &=& (1+\f{2\be(p-1)}{\Ga}) E_1 \la^{\f{2\be (p-1)}{\Ga}}.
 \end{eqnarray}
 Finally, there is the positivity of the (self-adjoint) linearized operator
 $$
  \cl f: =\dea f - \om f - p (c_{d,\ga} |\cdot|^{-\ga}*[\phi^{p-1} f])\phi^{p-1} -(p-1)  (c_{d,\ga} |\cdot|^{-\ga}*[\phi^{p}]) \phi^{p-2} f,
   $$
  on the co-dimension one subspace $\{\phi\}^\perp$. That is,
 \begin{equation}
 \label{69}
 \dpr{\cl h}{h}\geq 0,  h\in \{\phi\}^\perp.
 \end{equation}
  \end{theorem}
  \begin{proof}
  Fix $\la$. By scaling , one sees that the solution $\phi_\la$ of \eqref{30} can be represented by the following formula
  $$
  \phi_\la(x)=\la^{\f{\Ga+(p-1)d}{2\Ga}} \phi_1(\la^{\f{p-1}{\Ga}} x),
  $$
  From here, a short computation shows that it suffices to prove the results for the case $\la=1$.

   So, fix $\la=1$.
 Let $\phi=\phi_1$ be a minimizer for \eqref{30}.
  For any $\de>0$, consider $u_\de=\phi+\de h$, with $h$ real-valued. We have that
  $$
  E\left(\f{u_\de}{\|u_\de\|}\right)\geq E_1.
  $$
  Note that
  $$
  \|u_\de\|=\sqrt{\|\phi\|^2+2\de \dpr{\phi}{h}+O(\de^2)}=1+ \de \dpr{\phi}{h} +O(\de^2).
  $$
   We have
 $$
  \f{1}{2} \f{\|\naa u_\de\|^2}{\|u_\de\|^2} =   \f{1}{2} J - \de  (-\dpr{\dea \phi}{h}+J \dpr{\phi}{h}  ) +O(\de^2)
  $$
  and
  \begin{eqnarray*}
 & &  -\f{c_{d,\ga}}{2 p \|u_\de\|^{2 p}} \int_{\rtwo} \f{u_\de^{p}(x) u_\de^{p}(y)}{|x-y|^\ga } dx dy = \\
 &=&  - \f{1}{2 p} K +   \de \left[ \dpr{\phi}{h} K - \dpr{(c_{d,\ga} |\cdot|^{-\ga}*\phi^{p})\phi^{p-1}}{h} \right]  +O(\de^2).
  \end{eqnarray*}
  Taking into account that
  \begin{equation}
  \label{65}
   \f{1}{2} J  -  \f{1}{2 p} K=E_1,
   \end{equation}
    we conclude
  $$
  \de \dpr{\dea  \phi - (c_{d,\ga} |\cdot|^{-\ga}*\phi^{p})\phi^{p-1}+(K - J) \phi}{h}+O(\de^2)\geq 0.
  $$
  Since this is true for all $\de\in \rone$ and for all test functions $h$, we conclude that $\phi$ satisfies
  $$
  \dea \phi - (c_{d,\ga} |\cdot|^{-\ga}*\phi^{p})\phi^{p-1}+(K - J) \phi=0,
  $$
  which is the Euler-Lagrange equation \eqref{20}, with a scalar $\om=J-K$.
   Finally, there is the Pohozaev's identity, which we derive in the following way. Set
  $
  z_\mu(x)=\mu^{d/2} \phi(\mu x).
  $
  Clearly, since $\int_{\rd} z_\mu^2(x) dx=\int\phi^2(x) dx=1$, $z_\mu$  satisfies the constraint of \eqref{30}. Now
  \begin{eqnarray*}
  E(z_\mu) &=& \f{\mu^{2\be}}{2}  J - \mu^{\ga+d (p-2)} \f{1}{2 p} K.
  \end{eqnarray*}
  Since the scalar valued function $\mu\to E(z_\mu)$ achieves its minimum at $\mu=1$, we must have $\f{d E(z_\mu)}{d\mu}|_{\mu=1}=0$. This relation yields the Pohozaev's identity
  \begin{equation}
  \label{68}
  \be J - \f{\ga+d(p-2)}{2 p} K=0.
  \end{equation}
  Combining \eqref{65} and \eqref{68}, we obtain the formulas
  \begin{eqnarray*}
  K &=& \f{4\be p}{\Ga} (-E_1) \\
  J &=& \f{2(\ga+d(p-2))}{\Ga}(-E_1)\\
  \om &=& J-K= 2\f{2\be p-\ga-d(p-2)}{\Ga} E_1=(1+\f{2\be(p-1)}{\Ga}) E_1.
  \end{eqnarray*}
  Thus, we arrive at the statements of    \eqref{62}, \eqref{64}, \eqref{67}. Clearly, $K>0, J>0$, while $\om<0$, since $E_1<0$.

  We now establish the coercivity of $\cl$ on the co-dimension one subspace  $\{\phi\}^\perp$.  To that end, note that for every test function $h$,
  the function
  $$
  g(\de):= E\left(\f{\phi+\de h}{\|\phi+\de h\|}\right)
  $$
  has a minimum at $\de=0$. In fact, the Euler-Lagrange equation \eqref{20} is nothing but a rephrased version of the necessary condition for a minimum $g'(0)=0$.  Given that $g$ achieves its minimum at $\de=0$, one has  a second necessary condition for minimum, namely $g''(0)\geq 0$. We will exploit this fact to our advantage in order to deduce \eqref{69}.  In order to simplify the computations (and to reflect the fact that the coercivity of $\cl$ is only over $\{\phi\}^\perp$ anyway), we take $h: \dpr{h}{\phi}=0, \|h\|=1$. Note that under this restriction
  $$
  \|\phi+\de h\|=(1+\de^2)^{1/2}=1+\f{\de^2}{2} +O(\de^3).
  $$
  Next, taking into account that $\phi$ satisfies \eqref{20}, we write
    \begin{eqnarray*}
  g(\de) &=&   \f{1}{2} \f{\|\naa (\phi+\de h)\|^2}{\|\phi+\de h \|^2}   -\f{c_{d,\ga}}{2 p} \int_{\rtwo} \f{(\phi+\de h)^{p}(x)(\phi+\de h)^{p}(y)}{|x-y|^\ga \|\phi+\de h\|^{2 p}} dx dy=\\
  &=& g(0)+\f{\de^2}{2}\left[ \dpr{(\dea +K-J)h}{h} - p \dpr{(c_{d,\ga} |\cdot|^{-\ga}*[\phi^{p-1} h]\phi^{p-1}}{ h} \right]-\\
  &-& (p-1) \dpr{(c_{d,\ga} |\cdot|^{-\ga}*\phi^{p})\phi^{p-2} h}{h}  +  o(\de^2).
  \end{eqnarray*}
  Recall that $\om=J-K$. Since $g(\de)\geq g(0)$ for all small enough $\de$,  it follows that the operator $\cl$ defined by
  $$
  \cl f= \dea f - \om f - p (c_{d,\ga} |\cdot|^{-\ga}*\phi^{p-1} f)\phi^{p-1} -(p-1) (c_{d,\ga} |\cdot|^{-\ga}*[\phi^{p}]) \phi^{p-2} f
  $$
  satisfies $\dpr{\cl h}{h}\geq 0$, which is exactly \eqref{69}.
  \end{proof}
    \subsection{The linearized problem and  spectral properties of the self-adjoint part} We impose  the ansatz\footnote{We suppress the super index $\phi_\la$ in what follows, but we would like to keep $\phi_\la$ dependent upon the parameter $\la$}   $u=\phi_\la+\epsilon v$, where $v$ is necessarily complex valued field.  We have
  \begin{eqnarray*}
 & &  (|\cdot|^{-\ga}*|u|^p) |u|^{p-2} u =  (|\cdot|^{-\ga}*|\phi +\epsilon v|^p) |\phi+\epsilon v|^{p-2} (\phi+\epsilon v)= \\
  &=& (|\cdot|^{-\ga}*\phi^p) \phi^{p-1}+\epsilon[p(|\cdot|^{-\ga}*[\phi^{p-1}\Re v]) \phi^{p-1}+  (|\cdot|^{-\ga}*\phi^p) [(p-2) \phi^{p-2}\Re v+\phi^{p-2} v]]+\\
  & & + o(\epsilon).
  \end{eqnarray*}
Setting $u=e^{i \om t}[\phi +\epsilon v]=e^{i \om t}[\phi+\epsilon(\Re v+i \Im v)]$ in \eqref{10} and ignoring $o(\epsilon)$, we obtain the following linearized system
 $$
  \left|
  \begin{array}{l}
  -\p_t v_2 +\dea v_1-\om v_1 - p (c_{d,\ga} |\cdot|^{-\ga}*[\phi^{p-1} v_1])\phi^{p-1}-(p-1)(c_{d,\ga} |\cdot|^{-\ga}*\phi^{p}) \phi^{p-2} v_1 =0 \\
  \p_t v_1+\dea v_2-\om v_2 - (c_{d,\ga} |\cdot|^{-\ga}*\phi^{p})\phi^{p-2} v_2=0
  \end{array}
  \right.
  $$
  for $v_1=\Re v, v_2=\Im v$. As is customary, we adopt the notation
  \begin{eqnarray*}
  L_+ &=& \cl=\dea  - \om  - p (c_{d,\ga} |\cdot|^{-\ga}*\phi^{p-1} [\cdot])\phi^{p-1} -(p-1) (c_{d,\ga} |\cdot|^{-\ga}*[\phi^{p}]) \phi^{p-2}, \\
  L_- &=&  \dea -\om  - (c_{d,\ga} |\cdot|^{-\ga}*\phi^{p})\phi^{p-2}, \\
  \cj &=& \left(\begin{array}{cc} 0 & -1 \\ 1 & 0 \end{array} \right), L=\left(\begin{array}{cc} L_+ & 0 \\ 0 & L_-\end{array}\right).
  \end{eqnarray*}
  so that we can rewrite the eigenvalue problem \eqref{a:10} in the Hamiltonian form
  \begin{equation}
  \label{a:20}
  \vec{v}_t= \cj L \vec{v}.
  \end{equation}
  We will show that $L$ is a self-adjoint operator, at least for $p>2$.  Indeed, one can apply the KLMN theorem (see Theorem X.17 in \cite{ReSi}) for the operators
  $$ L_+ = \dea  - \om  -  p c_{d,\ga} V_1 - p(p-1) c_{d,\ga}V_2, \ , \ L_- = \dea -\om  - c_{d,\ga} V_2,$$
  where
  $$ V_1(f) = (|\cdot|^{-\ga}*\phi^{p-1} [f])\phi^{p-1}, \ \ V_2(f) =  (|\cdot|^{-\ga}*\phi^{p})\phi^{p-2} f. $$
  The check of the assumptions of the KLMN theorem follow from the simple Sobolev estimate
    $$
    \left|\langle |\cdot|^{-\ga}*g_1, g_2\rangle_{L^2} \right| = c  \left|\langle (-\Delta)^{-\alpha/2} g_1, (-\Delta)^{-\alpha/2} g_2\rangle_{L^2} \right| \leq C \|g_1\|_{L^{2d/(d+\alpha)}} \|g_2\|_{L^{2d/(d+\alpha)}},
    $$
    applied for $g_1=g_2= \phi^{p-1} f.$ In this way we find\footnote{recall that $\phi$ is a bell - shaped function}
    $$\left| \langle V_1(f),f \rangle_{L^2} \right| \leq C \|f\|_{L^2}. $$
    For the operator $V_2$ we observe that $ |\cdot|^{-\ga}*\phi^{p} \in L^\infty$, so we  find
        $$\left| \langle V_2(f),f \rangle_{L^2} \right| \leq C \|f\|_{L^2}. $$
We are in position to conclude that $L_\pm$ are self - adjoint operators, whence $L$ is self-adjoint as well.
      On the other hand,  $J$ is clearly skew-symmetric. By Weyl's criterion, both operators $L_\pm$,  have absolutely continuous spectrum, which fills the interval $[-\om, \infty)$, which verifies the spectral gap condition at zero, since $\om<0$ by virtue of Theorem \ref{theo:20}.
  In addition, we have verified in Theorem \ref{theo:20} that $L_+$ has at most  one negative eigenvalue. We have the following  lemma, regarding the spectral properties of $L$.
  \begin{lemma}
  \label{le:50}
 For $p>2$,  the self-adjoint operator $L_+=\cl$ has exactly one negative eigenvalue, while $L_-\geq 0$.
  \end{lemma}
As an immediate consequence of Lemma \ref{le:50}, the matrix operator $L$ has exactly one negative eigenvalue.
  \begin{proof}(Lemma \ref{le:50})
  Regarding $\cl$, we only need to verify that it does indeed have a negative eigenvalue. This is easily seen by testing the quantity $\dpr{\cl \phi}{\phi}$. Indeed, taking into account the Euler-Lagrange equation \eqref{20}, we compute
  \begin{eqnarray*}
  \cl \phi &=& \dea  \phi - \om \phi   - p (c_{d,\ga} |\cdot|^{-\ga}*\phi^{p})\phi^{p-1} -(p-1) (c_{d,\ga} |\cdot|^{-\ga}*[\phi^{p}]) \phi^{p-1} = \\
  &=&  -(2p-2)(c_{d,\ga} |\cdot|^{-\ga}*[\phi^{p}]) \phi^{p-1}.
  \end{eqnarray*}
  Thus,
  $$
  \dpr{\cl \phi}{\phi}=-(2p-2)c_{d,\ga} \int_{\rd} (|\cdot|^{-\ga}*[\phi^{p}]) \phi^{p}=-(2p-2)c_{d,\ga}  \int_{\rd\times \rd} \f{\phi^{p}(x) \phi^{p}(y)}{|x-y|^\ga} dx dy<0.
  $$
  Next, we wish to show that $L_-$ is a non-negative operator. Note that the Euler-Lagrange equation \eqref{20} is nothing but $L_-[\phi]=0$, where $\phi>0$. In particular zero is an eigenvalue for $L_-$ and it remains to show that it is at the bottom of its spectrum. Assume that this is not the case, hence $L_-$ has a negative eigenvalue, say $-\si^2$ and assume, without loss of generality that it is the smallest such eigenvalue.  In particular,
  \begin{equation}
  \label{a:30}
 -\si^2=\inf\limits_{\|\psi\|=1} \dpr{L_- \psi}{\psi}.
  \end{equation}
  The corresponding eigenfunction, say $\vp$ can be constructed as a minimizer of the minimization problem \eqref{a:30}. More precisely, upon introducing   the bell-shaped function \\ $V(x):=(c_{d,\ga} |\cdot|^{-\ga}*\phi^{p})\phi^{p-2}$, we have
  \begin{equation}
  \label{a:60}
  \left|
  \begin{array}{l}
  \dpr{L_- \psi}{\psi}= \|\naa \psi\|^2 -   \int_{\rd}  V(x) \psi^2(x) dx\to min \\
  \int_{\rd} \psi^2(x) dx=1.
  \end{array}
  \right.
  \end{equation}
  We will now show that the eigenfunction $\vp$ satisfies $\vp=\vp^*$ and as such is a positive function. To that end,  by Proposition \ref{prop:10}, we have
  $$
  \|\naa \psi\|^2\geq \|\naa [\psi^*]\|^2
  $$
  Next, applying \eqref{RR} and observing that $(h^2)^*=(h^*)^2$, we obtain
  \begin{equation}
  \label{a:50}
  \int_{\rd}  V(x) (\psi(x))^2 dx \leq \int_{\rd}  V(x) (\psi^*(x))^2  dx,
  \end{equation}
while the constraint $1=\int_{\rd} \psi^2(x) dx = \int_{\rd} (\psi^*(x))^2 dx$ remains satisfied. Thus, $\dpr{L_-\psi}{\psi}\geq \dpr{L_-\psi^*}{\psi^*}$.
It follows that the solution $\vp$ of \eqref{a:60}, which must exists, is bell-shaped and in particular $\vp>0$. But if such eigenfunction corresponds to a negative eigenvalue $-\si^2$, then it must be perpendicular to the eigenfunction $\phi$ corresponding to eigenvalue zero. However, both $\phi>0, \vp>0$, a contradiction. It follows that $L_-\geq 0$.

  \end{proof}
  At this point, we are essentially ready to consider the stability of these waves, more precisely the eigenvalue problem \eqref{a:20}. We will postpone  these  considerations to Section \ref{sec:4.6}. This is done in the interest of presenting an unified approach for the classical case of MVS waves and then for the fractional waves.  The approach for the fractional case turns out to be pretty similar, we outline the details in Section \ref{sec:4.6}.

  \section{Classification of the stability of the ground states for the Hartree and Klein-Gordon-Hartree models: proof of Theorems \ref{theo:60} and \ref{theo:50}}
  \label{sec:4}
We start with the proof of Theorem \ref{theo:60}. Henceforth,  the assumptions made in Theorem \ref{theo:60} are in force.  Recall
  $
  \Ga=2-\ga-(p-2)d.
  $
Consider the linearization of the solutions of the time-dependent Hartree model \eqref{t:10} around the ground states constructed in Theorem \ref{vs:10}
  and the Klein-Gordon-Hartree model \eqref{410}, around the ground state constructed in \eqref{420}.

\subsection{The linearized problem for the Hartree model \eqref{10}}
As before, we take the ansatz\footnote{We suppress the super index $\phi^\la$ in what follows, but we would like to keep $\phi^\la$ dependent upon the parameter $\la$}
  $$
  u=e^{-i t}[\vp +\epsilon v]=e^{-i t}[\vp +\epsilon(\Re v+i \Im v)]=e^{ - i  t}[\vp+\epsilon(v_1+i v_2)]
  $$
   in \eqref{10} and ignoring $O(\epsilon^2)$, leads us to the  following linearized system
  \begin{equation}
  \label{a:10}
  \left|
  \begin{array}{l}
  -\p_t v_2 -\De v_1+ v_1 - p I_\al[\vp^{p-1} v_1] \vp^{p-1}-(p-1) I_\al[\vp^{p}] \vp^{p-2} v_1 =0 \\
  \p_t v_1 - \De  v_2 + v_2 - I_\al[\vp^{p}] \vp^{p-2} v_2=0
  \end{array}
  \right.
  \end{equation}
  As is customary, we adopt the notation
  \begin{eqnarray*}
  L_+ &=& -\De +1  - pI_\al[\vp^{p-1} [\cdot]] \vp^{p-1} -(p-1) I_\al[\vp^{p}]  \vp^{p-2} \\
  L_- &=& -\De  +1  - I_\al[\vp^{p}] \vp^{p-2} \\
  \cj &=& \left(\begin{array}{cc} 0 & -1 \\ 1 & 0 \end{array} \right), L=\left(\begin{array}{cc} L_+ & 0 \\ 0 & L_-\end{array}\right).
  \end{eqnarray*}
  so that we can rewrite the eigenvalue problem \eqref{a:10} in the Hamiltonian form
 $$
  \vec{v}_t= \cj L \vec{v}
 $$
 Note that $\cj$ is clearly skew-symmetric. Next we  derive the linearized problem for the Klein-Gordon-Hartree model \eqref{410}.

  \subsection{Linearized problem for the Klein-Gordon-Hartree model \eqref{410}}
       As in the \\ Schr\"odinger case, take
       $$
       u=e^{i \om t}[\Psi+\epsilon v]=e^{i \om t}[\Psi+\epsilon(\Re v+i \Im v)]=e^{i \om t}[\Psi+\epsilon(v_1+i v_2)]
       $$
       and plug this in \eqref{410}. After ignoring $O(\epsilon^2)$ terms and taking real and imaginary parts, we arrive at
       \begin{equation}
       \label{b:15}
       \left|\begin{array}{l}
       \p_{tt} v_1 -2\om \p_t v_2+(1-\om^2) v_1 -\De v_1 -  p I_\al[\Psi^{p-1} v_1] \Psi^{p-1}-(p-1) I_\al[\Psi^{p}] \Psi^{p-2} v_1 =0,       \\
       \p_{tt} v_2 +2\om \p_t v_1 +(1-\om^2) v_2 -\De v_2 - I_\al[\Psi^{p}] \Psi^{p-2} v_2=0.
       \end{array}
       \right.
       \end{equation}
     In order to bring the eigenvalue problem \eqref{b:15} to a form similar to \eqref{a:10}, recall \eqref{eq:psi}. In accordance with that,
     we rescale the variables as follows $v_j(t,x)=e^{\la t \sqrt{1-\om^2}} V_j(x\sqrt{1-\om^2}), j=1,2$, so that the eigenvalue problem \eqref{b:15} is transformed into the standard form
      \begin{eqnarray*}
     & &  \la^2 V_1 -2\la \f{\om}{\sqrt{1-\om^2}}  V_2+L_+[V_1] =0,     \\
      & &  \la^2 V_2 +2\la \f{\om}{\sqrt{1-\om^2}} V_1 +  L_-[V_2]=0.
       \end{eqnarray*}
 Introducing the skew-symmetric matrix $\cj_\om:=\left(\begin{array}{cc} 0 &  - 2 \f{\om}{\sqrt{1-\om^2}} \\ 2 \f{\om}{\sqrt{1-\om^2}}  & 0 \end{array}\right)$, we can rewrite the relevant eigenvalue problem in the compact form
 \begin{equation}
 \label{b:20}
 \left(\begin{array}{cc} 0 & {\mathbf I_2} \\ - {\mathbf I_2} & -\cj_\om  \end{array}\right)  \left(\begin{array}{cc} L & 0 \\  0 &  {\mathbf I_2} \end{array}\right) \vec{W}=\la \vec{W}, \ \
 \vec{W}=  \left(\begin{array}{c}  \left(\begin{array}{c} V_1\\ V_2 \end{array}\right)  \\  \la \left(\begin{array}{c} V_1\\ V_2 \end{array}\right)   \end{array}\right).
 \end{equation}

\subsection{Spectral  information about the operators    $L_\pm$}
\label{sec:4.11}
Here, we shall need to summarize the results in \cite{MVS1} about the properties of the MVS solutions $\vp$ of \eqref{t:20}.
\begin{lemma}(Theorem 4 and Lemma 6.7, \cite{MVS1} )
\label{le:2}
The MVS solution $\vp$ of \eqref{t:20} satisfy
\begin{equation}
\label{n:10}
\lim_{|x|\to \infty} \f{I_\al[\vp^p]}{I_\al(x)}=\int_{\rd} \vp^p.
\end{equation}
For $p\geq 2$, $\vp$ has exponential decay at $\pm \infty$, while for $p<2$, there is the relation
\begin{equation}
\label{n:20}
 \vp(x) = \f{1}{|x|^{\f{d-\al}{2-p}}}\left(\f{\Ga(\f{d-\al}{2})}{\Ga(\f{\al}{2}) \pi^{d/2} 2^\al} \int_{\rd}  \vp^p\right)^{\f{1}{2-p}} + O\left(\f{1}{|x|^{\f{d-\al}{2-p}+2}}\right)
 \end{equation}
 for large $|x|$.
 In addition,
  \begin{eqnarray}
  \label{210}
 J[\vp] &:= &\int_{\rd} |\nabla \vp(x)|^2 dx = \f{d(p-2)+\ga}{2d-\ga-p(d-2)} \|\vp\|_{L^2}^2,   \\
  \label{220}
K[\vp] &:= & \dpr{I_\al[\vp^p]}{\vp^p}=  c_{d,\ga} \int_{\rd\times \rd} \f{\vp^{p}(x) \vp^{p}(y)}{|x-y|^\ga} dx dy=\f{2p}{2d-\ga-p(d-2)} \|\vp\|_{L^2}^2, \\
  \label{230}
  E[\vp] &=& \f{1}{2} J - \f{1}{2p} K= - \f{\Ga}{2(2d-\ga-p(d-2))} \|\vp\|_{L^2}^2.
  \end{eqnarray}
  {\bf Note:} By the assumptions in the existence theorem $2d-\ga-p(d-2)>0$, so   $sgn(E)=-sgn(\Ga)$.
\end{lemma}

  \begin{proof}
  The formula \eqref{n:10} appears in Theorem 4 in \cite{MVS1}. The formula \eqref{n:20} is a combination of the last statement of Theorem 4 and the final remark in the proof of Lemma 6.7. The formulas \eqref{210}, \eqref{220} and \eqref{230} are just an elementary consequence of the Pohozaev's identity \eqref{pohoz}, together with the relation
$\int_{\rd} [|\nabla u(x)|^2+|u(x)|^2] dx=\dpr{I_\al[|u|^p]}{|u|^p}$, which follows from \eqref{t:20} by taking dot product with $\vp$.
\end{proof}

 In addition to \eqref{n:10} and \eqref{n:20}, we will need more precise information on the behavior of $I_\al[\vp^p]$, for the case $p<2$. This is provided in the following lemma.
\begin{lemma}
Let $p\in  (2-\ga/d, 2)$. Then,
\label{le:vlado}
\begin{equation}
\label{n:30}
\f{I_\al[\vp^p]}{I_\al(x)}=[\int_{\rd} \vp^p+O(|x|^{-\min(2,\alpha)})]
\end{equation}
for large $|x|$.
\end{lemma}
{\bf Remark:} Note that the error term $O(|x|^{-\min(2,\alpha)})$ is not necessarily sharp for all   $p\in  (2-\ga/d, 2]$, but it is rather an upper bound, which suffices for our purposes.
  \begin{proof}
 We start with  the relation
  $$
  |x-y|^{-\gamma} = |x|^{-\gamma}\left(1-\gamma \frac{\langle x, y \rangle}{|x|^2} + O(|y|^2/|x|^2) \right).
  $$
  Using $ \int_{|y| \leq |x|/2} y \varphi^p|y)dy =0$,  we get
  \begin{equation}
  \label{po}
  \int_{|y|\leq |x|/2} \frac{\varphi^p(|y|)dy}{|x-y|^\gamma} = |x|^{-\ga} \int_{|y| \leq |x|/2} \vp^p(|y|) dy +
  O\left( \f{1}{|x|^{2+\ga}}\int_{|y|\leq |x|/2} |y|^2 \varphi^p(y)dy \right).
  \end{equation}
  Note that for large $y$,
  $$
   |y|^2 \varphi^p(y)\leq \f{C}{|y|^{\f{p\ga}{2-p}-2}},
  $$
  which is integrable,  provided $\f{p\ga}{2-p}-2>d$ or $p > 2- 2\ga/(d+2+\ga)$. Thus, we have that the error term in \eqref{po} is $O(|x|^{-2-\ga})$.

  On the other hand, if $2-\f{\ga}{d}<p\leq 2- 2\ga/(d+2+\ga)$, we may estimate
  $$
   \f{1}{|x|^{2+\ga}}\int_{|y|\leq |x|/2} |y|^2 \varphi^p(y)dy\leq \f{C}{|x|^{\ga+\al}}\int_{|y|\leq |x|/2} |y|^\al \varphi^p(y)dy.
  $$
  Since for large $y$, we have
  $$
  |y|^\al \varphi^p(y)\leq \f{C}{|y|^{\f{p\ga}{2-p}-\al}},
  $$
  which is integrable when $\f{p\ga}{2-p}-\al>d$ or $p>\f{2(d+\al)}{\ga+\al+d}=2-\f{\ga}{d}$, which is the full range of interest according to Theorem \ref{vs:10}.  Thus, the error term in \eqref{po} is now $O(|x|^{-\ga-\al})$. This finishes the proof of \eqref{n:30} and Lemma \ref{le:vlado}.
  \end{proof}

\noindent Our next lemma provides the self-adjointness of $L_\pm$ as well as a description of the absolutely continuous spectrum.
\begin{lemma}
\label{le:09}
 The linearized operators $L_\pm$  with domains $D(L_\pm)=H^2(\rd)$,  are self-adjoint. In addition, for $p\geq 2$, $\si_{a.c.}(L_+)=\si_{a.c.}(L_-)=[1, \infty)$. For $p<2$ however, $\si_{a.c.}(L_-)=[0, \infty)$, while $\si_{a.c.}(L_+)=[2-p, \infty)$.
\end{lemma}
\begin{proof}
Let us first go through the easy cases $p\geq 2$. The self-adjointness in this case is an easy matter, since  one can apply the KLMN theorem (see Theorem X.17 in \cite{ReSi}) for
the operators
  $$
  L_+ = -\De   +1  -   p V_1 - (p-1)V_2, \ , \ L_- =  -\De  +1  -  V_2,
  $$
  where
  $$
  V_1(f) =  c_{d,\ga} (|\cdot|^{-\ga}*\vp^{p-1} [f])\vp^{p-1}, \ \ V_2(f) =  c_{d,\ga} (|\cdot|^{-\ga}*\vp^{p})\vp^{p-2} f.
  $$
 We need to estimate $\dpr{V_j f}{f}, j=1,2$. We have
    $$
     \left|\langle |\cdot|^{-\ga}*g_1, g_2\rangle  \right| = c  \left|\langle (-\Delta)^{-\alpha/4} g_1, (-\Delta)^{-\alpha/4} g_2\rangle  \right| \leq C \|g_1\|_{L^{2d/(d+\alpha)}} \|g_2\|_{L^{2d/(d+\alpha)}}.
    $$
Applying this to $g_1=g_2=\vp^{p-1} f$ yields
$$
\dpr{V_1 f}{f}\leq C \|f\|^2 \|\vp^{p-1}\|_{L^{\f{2d}{\al}}}^2 = C \|f\|_{L^2}^2 \|\vp\|_{L^{\f{2d(p-1)}{\al}}}^{2(p-1)}.
$$
Since by construction $\vp \in L^q, q\in [2, \infty]$ and\footnote{recall that by assumption for the existence of $\vp$: $p>1+\f{\al}{d}$}   $\f{2d(p-1)}{\al}>2$,   we conclude that $V_1$ satisfies the requirements of KLMN theorem and it is an
admissible perturbation\footnote{This argument actually works for all $p>1+\f{\al}{d}$ and it is not limited to $p\geq 2$. }  of the self-adjoint operator $-\De-\om$.

Regarding $V_2$, we have by Hardy-Litllewood-Sobolev inequality
$$
|\dpr{V_2 f}{f}| \leq C \||\cdot|^{-\ga}*\vp^p\|_{L^\infty} \|\vp\|_{L^\infty}^{p-2} \|f\|_{L^2}^2\leq C \|\vp\|_{L^{\f{pd}{\al}}}^{p} \|f\|_{L^2}^2.
$$
This is also enough by KLMN, since $\f{pd}{\al}\geq \f{2d}{\al}>2$. Thus, the self-adjointness of $L_\pm$ in the case $p\geq 2$ follows by KLMN.
The argument for $V_2$  however is limited to $p\geq 2$, because  otherwise $\vp^{p-2}$ is actually unbounded as $|x|\to \infty$ and the argument above clearly fails.

Assume $p<2$. Let us consider first the self-adjointness of $L_-$. Since $L_-=-\De+1-I_\al[\vp^p]\vp^{p-2}$, clearly this is not a potential that decays at $\infty$. In fact, we will show that
\begin{equation}
\label{600}
I_\al[\vp^p]\vp^{p-2}=1+O(|x|^{-\min(2,\alpha)}).
\end{equation}
Indeed, according to   \eqref{n:20}
$$
\vp^{p-2} (x)=\f{|x|^{d-\al} \Ga(\al/2) \pi^{d/2} 2^\al}{\Ga((d-\al)/2)\int \vp^p}[1+O(|x|^{-\min(2,\alpha)}))].
$$
 Thus, according to \eqref{n:10},
$$
I_\al[\vp^p]\vp^{p-2}(x)= \f{I_\al[\vp^p]}{I_\al(x) \int \vp^p}[1+O(|x|^{-\min(2,\alpha)})]=1+O(|x|^{-\min(2,\alpha)}).
$$
It follows that
\begin{equation}
\label{n:50}
L_-=-\De+1-I_\al[\vp^p]\vp^{p-2}=-\De+G(x),
\end{equation}
where $G$ is a smooth and bounded  function, with $|G(x)|\leq C  (1+|x|^{\min(2,\alpha)})^{-1}$.
It follows that $L_-=L_-^*$. In addition, $\si_{ess}(L_-)=[0, \infty)$.

Regarding $L_+$, we have that for $p<2$,
\begin{equation}
\label{n:60}
L_+=-\De+1 - pV_1 - (p-1) V_2 = -\De+ (2-p) - p V_1 - (p-1) G(x),
\end{equation}
whence $L_+$ is also self-adjoint, with $\si_{ess}(L_+)=[2-p, \infty)$.
\end{proof}
Next, we discuss the point spectrum of the operators $L_\pm$. We have the following result.
\begin{lemma}
  \label{le:76}
The linearized operator $L_+$
  has exactly one negative eigenvalue.  On the other hand, $L_-\geq 0$ with an eigenvalue at zero.  The eigenvalue at zero is simple, with eigenfunction $\vp$.
     \end{lemma}
  \begin{proof}
Let us first establish the claims regarding $L_-$.   Clearly $L_-[\vp]=0$, as this is simply \eqref{t:20}. Thus zero is an eigenvalue for $L_-$, with eigenfunction $\vp$.  Assuming that $L_-$ has a negative eigenvalue will lead to a contradiction. Indeed, pick the bottom of the spectrum for $L_-$. By the results in Lemma \ref{le:09} (and the description of the structure of $L_-$),  it will necessarily have a positive eigenfunction, say $\psi_0$. But then, $\psi_0\perp \vp$ as eigenfunctions corresponding to different eigenvalues, a contradiction. Thus, zero is the bottom of the spectrum.  The simplicity of the bottom of the spectrum (in this case the zero eigenvalue)  is also well-known  by the  Sturm oscillation argument.

  We now turn our attention to $L_+$.  First, it is easy to see that
  \begin{eqnarray*}
  L_+ \vp &=& -\De \vp +\vp   - pI_\al[\vp^{p}]\vp^{p-1} -(p-1) I_\al[\vp^{p}] \vp^{p-1} = \\
  &=&  -(2p-2) I_\al[\vp^{p}]  \vp^{p-1}.
  \end{eqnarray*}
  Thus,
  $$
  \dpr{L_+\phi}{\phi}=-(2p-2) \dpr{I_\al[\vp^{p}]}{\vp^p}<0.
  $$
  This shows that $L_+$ has at least one negative eigenvalue.
   It remains to show that $L_+$ is positive on a codimension one subspace.

 To that end,  consider the minimizer $\Phi$ of the optimization problem \eqref{200}. Consider a perturbation of $\Phi$ in the form $u_\epsilon=\Phi+\epsilon h$, for a real-valued  function $h$  and $\epsilon: |\epsilon|<<1$. Expanding in Taylor series up to order $\epsilon^2$, we obtain
  \begin{eqnarray*}
& &   I[u] :=  \dpr{\nabla u}{\nabla u}+ \dpr{u}{u}= I[\Phi]+2\epsilon \dpr{-\De \Phi+\Phi}{h}+\epsilon^2 \dpr{-\De h+h}{h};\\
  & & M[u] = \dpr{I_\al[|u|^p]}{|u|^p}=M[\Phi]+\\
  &+& \epsilon[2p \dpr{I_\al[\Phi^p]\Phi^{p-1}}{ h}]+\epsilon^2[p^2\dpr{I_\al[\Phi^{p-1}h]}{\Phi^{p-1}h}+p(p-1)\dpr{I_\al[\Phi^p] \Phi^{p-2} h}{h}] +o(\epsilon)
  \end{eqnarray*}
Denote $I_0:=I[\Phi], M_0:=M[\Phi]$. Recall that since $\Phi$ is a minimizer for \eqref{200}, we will have
\begin{equation}
\label{300}
g(\epsilon):=\f{I[u_\epsilon]}{(M[u_\epsilon])^{1/p}}\geq \frac{I_0}{(M_0)^{1/p}}=g[0].
\end{equation}
 Clearly, such a relation implies that zero is a minimum for the function $\epsilon\to g(\epsilon)$. In particular, $g'(0)=0$. This is exactly the Euler-Lagrange equation for \eqref{200}, which means that $\Phi$ satisfies the PDE
 \begin{equation}
 \label{310}
 -\De \Phi+\Phi - \f{I_0}{M_0} I_\al[\Phi^p]\Phi^{p-1}=0.
 \end{equation}
Clearly, a function in the form $\vp=t_0\Phi$ will satisfy \eqref{t:20}, once we subject it
 to the normalization $\int_{\rd} [|\nabla \vp(x)|^2+|\vp(x)|^2] dx=\dpr{I_\al[\vp^p]}{\vp^p}$. That is
 \begin{equation}
 \label{320}
 \vp=\left(\f{I_0}{M_0}\right)^{\f{1}{2p-2}}\Phi.
  \end{equation}
We take advantage of the variational structure to show  that the operator $L_+$ has exactly one negative eigenvalue. Indeed, since $g$ has absolute minimum at $\epsilon=0$ and $g'(0)=0$, it is necessary that $g''(0)\geq 0$. In order to simplify the computations, let us take a function $h$, which satisfies the orthogonality condition $\dpr{I_\al[\Phi^p]\Phi^{p-1}}{ h}=0$. That is $h\perp I_\al[\Phi^p]\Phi^{p-1}$. By \eqref{310}, it follows that $\dpr{-\De \Phi+\Phi}{h}=0$ as well. It is now easy to see that the following expansion in powers of $\epsilon$ holds
\begin{eqnarray*}
& & g(\epsilon) =  g(0)+ \\
&+& \f{\epsilon^2}{M_0^{1/p}}\left[\dpr{-\De h+h}{h} - \f{I_0}{M_0}[p\dpr{I_\al[\Phi^{p-1}h]}{\Phi^{p-1}h}+ (p-1)\dpr{I_\al[\Phi^p] \Phi^{p-2} h}{h}] \right]+ o(\epsilon^2).
\end{eqnarray*}
We see that  $g''(0)\geq 0$ is equivalent to the positivity of the operator\footnote{where we have used the relation \eqref{320}}
  $$
 -\De+1  - \f{I_0}{M_0}[p I_\al[\Phi^{p-1}\cdot]\Phi^{p-1} + (p-1) I_\al[\Phi^p] \Phi^{p-2}]=-\De+1  - p I_\al[\vp^{p-1}\cdot]\vp^{p-1} - (p-1) I_\al[\vp^p] \vp^{p-2}
  $$
  on the subspace $\{I_\al[\Phi^p]\Phi^{p-1}\}^\perp$. We conclude that
  $$
  L_+|_{\{I_\al[\vp^p]\vp^{p-1}\}^\perp}\geq 0,
  $$
  as claimed.

    \end{proof}

  \subsection{The basics of the  instabilities  index counting}
  \label{sec:4.2}
  Now that we have established Lemma \ref{le:76}, we are ready to discuss the spectral stability of the waves $e^{-i  t} \vp$. In fact, the eigenvalue problem \eqref{a:20} falls within the scope\footnote{In the standard formulation, the GSS theory requires that there is a spectral gap between the zero and the continuous spectrum of $L$. In our case, this is clearly violated in the case $p<2$, since $\si_{a.c.}[L_-]=[0, \infty)$. By a remark in the argument in the original paper, this situation is also covered, in other words if the continuous spectrum just touches the zero, the statement still goes through as in the case with a spectral gap. For further justification in this case of touching, one should consult \cite{LLe} as well.}  of the Grillakis-Shatah-Strauss (GSS) theory, \cite{GSS}, see also \cite{KKS, KKS2}.  Recall that we have established that $L$ has exactly
  one negative eigenvalue\footnote{We henceforth adopt the notation $n(S)$ for a number of strictly negative eigenvalues of a self-adjoint operator/matrix $S$}, $n(L)=1$. In principle, in order to apply the theory, one needs to identify the kernel of the operator $L$. We have already know quite a bit about it - $\vp \in Ker[L_-]$, while a differentiation of the Euler-Lagrange equation \eqref{20},  in each of the variables $x_1, \ldots, x_d$ shows that $\f{\p \vp}{\p x_j}, j=1, \ldots, d$ is in the kernel of
  $\cl=L_+$, i.e. $span \{ \f{\p \vp}{\p x_j}, j=1, \ldots, d\}\subset Ker[\cl]$.  An important problem in the theory has been to determine whether these are indeed all of the linearly independent elements of $Ker[\cl]$, that is - is it true that
  \begin{equation}
  \label{a:80}
  Ker[\cl]=span \{ \f{\p \vp}{\p x_j}, j=1, \ldots, d\}?
  \end{equation}
  Ground states with the property \eqref{a:80} has been referred to as {\it non-degenerate}, \cite{FL}, \cite{MVS1, MVS2, MVS3}. Our argument goes forward even without knowledge of the non-degeneracy\footnote{although such a statement is very likely to hold} of $\vp$. By the GSS theory, we have that if $Ker[L]=span\{y_j, j=1, \ldots, l\}$ and $\cj$ is invertible with \\  $\cj^{-1}: Ker[L]\to [Ker[L]]^\perp$,  then
  \begin{equation}
  \label{a:100}
  \# \{\la: \Re\la>0: \cj L f=\la f\}=n(L)-n(D)=1-n(D),
  \end{equation}
  where the matrix $D\in M_{l,l}$ has entries
    \begin{equation}
   \label{a:105}
  D_{ij}=\dpr{L\psi_j }{\psi_i}: L\psi_j = \cj^{-1} y_j,
  \end{equation}
  where the equation $L\psi_j =\cj^{-1} y_j$ has a  solution\footnote{which is not unique, unless $Ker[L]=\{0\}$} $\psi_j$, since $\cj^{-1}: Ker[L]\to [Ker[L]]^\perp$.
   \subsection{Classification of the stability  for the Hartree solitary waves - proof of Theorem \ref{theo:60}}
   \label{sec:4.5}
   We start our considerations with a calculation, that will be useful in the sequel.
   \begin{lemma}
   \label{le:n10}
   $\vp\perp Ker[L_+]$ and moreover,
   $$
   \dpr{L_+^{-1} \vp}{\vp}=-\f{\Ga}{4(p-1)} \|\vp\|_{L^2}^2.
   $$
   \end{lemma}
   \begin{proof}
  We take advantage of the scaling of the PDE \eqref{t:20}. More precisely, introduce $\vp_\la$, so that
  \begin{equation}
  \label{390}
  \vp=\la^b \vp_\la(\la x),
  \end{equation}
  where $\la>0$ and  $b$ is a parameter to be determined from the scaling.   Plugging this into \eqref{t:20}, we find
  $$
  - \la^{b+2} \De \vp_\la+\la^b \vp_\la - \la^{b(2p-1)-(d-\ga)}   I_\al[\vp_\la^p] \vp_\la^{p-1} =0.
  $$
 Equating the powers $b+2=b(2p-1)-(d-\ga)$ yields $b=\f{2+d-\ga}{2(p-1)}$, which we use henceforth. Dividing by $\la^{b+2}$ yields the relation
 \begin{equation}
 \label{400}
 -   \De \vp_\la+\la^{-2} \vp_\la -  I_\al[\vp_\la^p] \vp_\la^{p-1} =0
 \end{equation}
  Taking a derivative in $\la$ in \eqref{400} yields
 $$
 -   \De[\f{\p \vp_\la}{\p \la}]+\la^{-2} [\f{\p \vp_\la}{\p \la}] -2\la^{-3} \vp_\la -  p I_\al[\vp_\la^{p-1} \f{\p \vp_\la}{\p \la}] \vp_\la^{p-1} - (p-1)  I_\al[\vp_\la^p] \vp_\la^{p-2} [\f{\p \vp_\la}{\p \la}]=0
$$
  Evaluating the previous expression at $\la=1$ can be interpreted as follows
  $$
  L_+[\f{\p \vp_\la}{\p \la}|_{\la=1}]=2\vp.
  $$
  this shows in particular that $\vp\perp Ker[L_+]$. In addition,
  $$
  \dpr{L_+^{-1} \vp}{\vp}= \f{1}{2} \dpr{\f{\p \vp_\la}{\p \la}|_{\la=1}}{\vp}=\f{1}{4}  [\p_\la  \|\vp_\la\|^2 ]|_{\la=1}=\f{d-2b}{4} \|\vp\|_{L^2}^2=-\f{\Ga}{4(p-1)} \|\vp\|_{L^2}^2.
  $$
   \end{proof}
  We are now in a position to consider the eigenvalue problem for the Hartree problem \eqref{a:20}. With the assignment, $\vec{v}\to e^{\la t} \vec{v}$, we are led to consider
   \begin{equation}
   \label{a:201}
   \cj L \vec{v} = \la \vec{v}.
   \end{equation}
   We know that $Ker[L_-]=span[\vp]$, while\footnote{As it was explained above in Section \ref{sec:4.2}, $Ker[L_+]$ contains at least the vectors $\p_j \vp, j=1, \ldots, d$}  $Ker[L_+]=span[y_1, \ldots, y_l]$. By Lemma \ref{le:n10}, $Ker[L_-]\perp Ker[L_+]$, whence
   \begin{equation}
   \label{opl}
   \cj^{-1} (Ker[L])= \cj^{-1} \left[ \begin{array}{c} Ker[L_+] \\ Ker[L_-] \end{array} \right] = \left[ \begin{array}{c} Ker[L_-] \\ Ker[L_+] \end{array} \right]\perp
   \left[ \begin{array}{c} Ker[L_+] \\ Ker[L_-] \end{array} \right]=Ker[L],
   \end{equation}
   whence $\cj^{-1}(Ker[L])\subset (Ker[L])^\perp$ as required. In addition, $L_-^{-1}$ is positive definite matrix on $span[y_1, \ldots, y_l]\subset (Ker[L_-])^\perp$. Thus, the matrix $D$, introduced in \eqref{a:105} has at most one  negative eigenvalue. Moreover, there is a negative eigenvalue if and only if
   $$
   D_{11}=\dpr{L^{-1} \cj^{-1}\left(\begin{array}{c} 0 \\ \vp \end{array} \right)}{\cj^{-1}\left(\begin{array}{c} 0 \\ \vp \end{array} \right)}=\dpr{L_+^{-1} \vp}{\vp}<0
   $$
 This, together with \eqref{a:100} allows us to derive a Vakhitov-Kolokolov type criteria for the Hartree waves, namely that stability of $e^{- i t} \vp$ is equivalent to
 $\dpr{L_+^{-1} \vp}{\vp}<0$. Using the formula for $\dpr{L_+^{-1} \vp}{\vp}$ in Lemma \ref{le:n10}, we conclude that the stability occurs exactly when $\Ga>0$. Moreover, if $\Ga<0$, there is a pair (one  positive and one negative) of eigenvalues $\pm\la$ in \eqref{a:201}. By the continuity of the spectrum on $\Ga$,  we have that  for $\Ga=0$, the pair $\pm 0$ transitions through the zero to become a pair of purely imaginary eigenvalues, so the eigenvalue problem has an extra pair of generalized eigenvalues at zero, when $\Ga=0$.

   \subsection{Classification of the stability for the waves  $e^{i \om t} \Psi_\om$ for the
   Klein-Gordon-Hartree model: Proof of Theorem \ref{theo:50}}

Much of what was established in Section \ref{sec:4.11} will be useful for the Klein-Gordon case as well.
Indeed for the eigenvalue problem \eqref{b:20}, we have
\begin{eqnarray*}
\left(\begin{array}{cc} 0 & {\mathbf I_2} \\ - {\mathbf I_2} & -\cj_\om  \end{array}\right)^{-1}
Ker\left(\begin{array}{cc} L & 0 \\  0 &  {\mathbf I_2} \end{array}\right)= \left(\begin{array}{c} \cj_\om[Ker[L]] \\ 0 \end{array}\right) \perp \left(\begin{array}{c} \ Ker[L] \\ 0 \end{array}\right)
\end{eqnarray*}
as established in \eqref{opl}. This is one of the requirements of GSS and it has now been verified. Similar to the arguments in Section \ref{sec:4.5}, the portion of the matrix $D$ generated by $Ker[L_+]$ is trivially positive definite. Indeed, consider the elements of $Ker[L]$ in the form
$
s_j:=\left(\begin{array}{c} y_j \\ 0\\ 0\\ 0  \end{array} \right), j=1, \ldots, l.
$
 Then, an easy computation shows that
 \begin{eqnarray*}
 D_{i,j}&=& \dpr{\left(\begin{array}{cc} L & 0 \\  0 &  {\mathbf I_2} \end{array}\right)^{-1}\left(\begin{array}{cc} 0 & {\mathbf I_2}
 \\ - {\mathbf I_2} & -\cj_\om  \end{array}\right)^{-1}  s_j}{\left(\begin{array}{cc} 0 & {\mathbf I_2} \\ - {\mathbf I_2} & -\cj_\om  \end{array}\right)^{-1}  s_i}=\\
 &=& \dpr{\left(\begin{array}{cc} L & 0 \\  0 &  {\mathbf I_2} \end{array}\right)^{-1}\left(\begin{array}{cc} -\cj_\om & -{\mathbf I_2}
 \\ {\mathbf I_2} & 0 \end{array}\right)  s_j}{\left(\begin{array}{cc} -\cj_\om & -{\mathbf I_2} \\ {\mathbf I_2} & 0  \end{array}\right)  s_i}=\\
 &=&  \dpr{\left(\begin{array}{cc} L^{-1}  & 0 \\  0 &  {\mathbf I_2} \end{array}\right) \left(\begin{array}{c}  \f{2\om}{\sqrt{1-\om^2}}\left(\begin{array}{c} 0 \\ -y_j\end{array}\right) \\  \left(\begin{array}{c} y_j \\ 0 \end{array}\right) \end{array}\right)} {\left(\begin{array}{c}  \f{2\om}{\sqrt{1-\om^2}} \left(\begin{array}{c} 0 \\ -y_j\end{array}\right) \\  \left(\begin{array}{c} y_j \\ 0 \end{array}\right) \end{array}\right)}\\
 &=& \f{4\om^2}{1-\om^2} \dpr{L_-^{-1} y_i}{y_j}+\|y_j\|^2.
 \end{eqnarray*}
The claim about the positivity of the portion of the matrix corresponding to these eigenvectors follows from  the positivity of $L_-^{-1}$ on $span[y_1, \ldots y_d]\subset Ker[L_-]^\perp$, which was previously established. Once again, the stability is found to be equivalent to the following criteria
   $$
   D_{11}=\dpr{\left(\begin{array}{cc} L & 0 \\  0 &  {\mathbf I_2} \end{array}\right)^{-1}\left(\begin{array}{cc} 0 & {\mathbf I_2}
 \\ - {\mathbf I_2} & -\cj_\om  \end{array}\right)^{-1}  s_0}{\left(\begin{array}{cc} 0 & {\mathbf I_2} \\ - {\mathbf I_2} & -\cj_\om  \end{array}\right)^{-1}  s_0}<0,
   $$
   where $s_0:=\left(\begin{array}{c}  0 \\ \vp\\ 0\\ 0  \end{array} \right)$. So, it remains to compute $D_{11}$. We have
   \begin{eqnarray*}
   D_{11} &=&     \dpr{\left(\begin{array}{cc} L^{-1}  & 0 \\  0 &  {\mathbf I_2} \end{array}\right) \left(\begin{array}{c}  \f{2\om}{\sqrt{1-\om^2}}\left(\begin{array}{c} \vp \\ 0 \end{array}\right) \\  \left(\begin{array}{c} 0 \\ \vp \end{array}\right) \end{array}\right)} {\left(\begin{array}{c}  \f{2\om}{\sqrt{1-\om^2}} \left(\begin{array}{c} \vp  \\ 0 \end{array}\right) \\  \left(\begin{array}{c} 0 \\ \vp \end{array}\right) \end{array}\right)}=\\
   &=&  \f{4\om^2}{1-\om^2} \dpr{L_+^{-1} \vp}{\vp}+\|\vp\|^2=\|\vp\|^2\left(1- \f{\Ga \om^2}{(p-1)(1-\om^2)}\right),
   \end{eqnarray*}
   where in the last equality, we have used the formula for $\dpr{L_+^{-1} \vp}{\vp}$ established in Lemma \ref{le:n10}.
Since $\om \in (-1,1)$, we have instability ($D_{11}>0$), if $\Ga<0$. If $\Ga>0$, we solve the inequality $1- \f{\Ga \om^2}{(p-1)(1-\om^2)}<0$ to obtain the necessary and sufficient condition for stability
$$
1>|\om|>\sqrt{\f{p-1}{p-1+\Ga}}=\sqrt{\f{p-1}{2+\al-(p-1)(d-1)}}.
$$

\subsection{On the stability of the ``normalized'' waves for the fractional problem: Proof of Theorem \ref{a:110}}
\label{sec:4.6}
In order to establish the stability of the waves, we consider the eigenvalue problem \eqref{a:20},
with the assignment $\vec{v}\to e^{\la t} \vec{v}$. It now reads
\begin{equation}
\label{vl:50}
\cj L \vec{v}= \la \vec{v}.
\end{equation}
It was already established that $L$ is self-adjoint, at least for $p>2$ and in addition, according to Lemma \ref{le:50}, $n(L_+)=1$,
while $L_-\geq 0$, with a simple eigenvalue at zero spanned by $\phi$. Thus, we apply the instabilities index count. In fact formulas \eqref{a:100}, \eqref{a:105} apply, in addition to \eqref{opl}, which shows that $\cj(Ker[L])\subset Ker[L]^\perp$. Note that we still do not know that $\phi$ is non-degenerate, that is whether one has more than
$\p_1 \phi, \ldots, \p_d \phi$ in $Ker[L]$ or not. Similarly to the classical case, we may sidestep this issue by the positivity of $L_-^{-1}$ on
$(Ker[L_-])^\perp$. Hence, stability is reduced to the sign of the quantity $D_{11}$. More specifically,   since $\phi\in Ker[L_-]$,
we may label $y_1:=\left(\begin{array}{c} 0 \\ \phi \end{array}\right)\in Ker[L]$. Then, {\it if it is indeed the case that $\phi\in (Ker[\cl])^\perp$}, we can compute
$$
 \dpr{D y_1}{y_1}= D_{11}=\dpr{\cl^{-1}[\phi]}{\phi}.
  $$
  If we establish $\dpr{\cl^{-1}[\phi]}{\phi}<0$, this would be enough to claim that $D$ has a negative eigenvalue and we are done, since
  \eqref{a:100} predicts $1-1=0$ eigenvalues of \eqref{vl:50} with positive real part.

  For the computation of $\dpr{\cl^{-1}[\phi]}{\phi}$, we can apply scaling argument similar to the one presented in Lemma \ref{le:n10}.
Instead, we take derivative in $\la$ in the Euler-Lagrange equation \eqref{20}. We obtain
   $$
   \cl\left[\f{\p \phi}{\p \la}\right]=\f{\p \om}{\p \la} \phi.
   $$
   Thus $\phi\in (Ker[\cl])^\perp$ and
   $$
   \dpr{\cl^{-1}\phi}{\phi}=\f{1}{2 \f{\p \om}{\p \la}} \p_\la \|\phi^\la\|_{L^2}^2.
   $$
   By \eqref{55}  we have
   $
   \|\phi^\la\|_{L^2}^2=\la\|\phi^1\|_{L^2}^2,
   $
   whence $\p_\la \|\phi^\la\|_{L^2}^2=\|\phi^1\|_{L^2}^2>0$.  According  to  \eqref{67},
   $$
   \f{\p \om}{\p \la}=(1+\f{2\be(p-1)}{\Ga}) \f{2\be(p-1)}{\Ga} I_1 \la^{\f{2\be(p-1)}{\Ga}-1}<0,
   $$
   since $I_1<0$, $\Ga>0$.      We have thus proved Theorem \ref{a:110}. Note that the restriction $p>2$ appeared only to satisfy technical (but important) requirements for self-adjointness of $L$ and it is not necessary in the index computations. Thus, we expect this to be a removable, technical assumption, once we have more information about the waves $\phi$ similar to the ones in the classical case, e.g. Theorem \ref{vs:10}.


\end{document}